\documentclass[12pt, oneside, reqno]{article}
\usepackage{amsmath,amsthm}
\usepackage{amssymb,latexsym}
\usepackage{enumerate}
\usepackage{cases}

\newtheorem{thm}{Theorem}[section]

\newtheorem{lem}[thm]{Lemma}

\theoremstyle{definition}

\numberwithin{equation}{section}

\begin{document}

\title{\Large Representation functions in the set of natural numbers}
\author{Shi-Qiang Chen\thanks{School of Mathematics and Statistics,
Anhui Normal University,  Wuhu  241003,  P. R. China. Email: csq20180327@163.com.
This author is supported by the National Natural Science Foundation of China, Grant No.
12301003, the Anhui Provincial Natural Science Foundation (Grant No. 2308085QA02) and
the University Natural Science Research Project of Anhui Province (Grant No. 2022AH050171).},
 Csaba
S\'andor \thanks{Department of Stochastics, Institute of Mathematics and Department of Computer Science and Information Theory, Budapest University of
Technology and Economics, M\H{u}egyetem rkp. 3., H-1111 Budapest, Hungary; MTA-BME Lend\"ulet Arithmetic Combinatorics Research Group, ELKH, M\H{u}egyetem rkp. 3., H-1111 Budapest, Hungary. Email: csandor@math.bme.hu. This author was supported by the NKFIH Grants No. K129335.},
Quan-Hui Yang \thanks{School of Mathematics and Statistics, Nanjing University of Information Science and Technology, Nanjing 210044, China. Email: yangquanhui01@163.com. This author is supported by the National Natural Science Foundation of China, Grant No.
12371005. }
}
\date{}
\maketitle \baselineskip 18pt

 {\bf Abstract:} Let $\mathbb{N}$ be the set of all nonnegative integers. For $S\subseteq \mathbb{N}$ and $n\in \mathbb{N}$,
let $R_S(n)$ denote the number of  solutions of the equation
$n=s+s'$, $s, s'\in S$, $s<s'$.
 In this paper, we determine the structure of all sets $A$ and $B$ such that $A\cup B=\mathbb{N}\setminus\{r+mk:k\in\mathbb{N}\}$, $A\cap B=\emptyset$ and $R_{A}(n)=R_{B}(n)$ for every positive integer $n$, where $m$ and $r$ are two integers with $m\ge 2$ and $r\ge 0$.



\section{Introduction}

~~~Let $\mathbb{N}$ be the set of all nonnegative integers.  For $S
\subseteq \mathbb{N}$  and $n\in \mathbb{N}$, the representation function
 $R_1(S,n)$, $R_2(S,n)$, $R_3(S,n)$ denote
the number of representations of $n$ in the form:
(1) $n=s+s'$, $s, s'\in S$, (2) $n=s+s'$, $s<s'$, $s, s'\in S$, (3) $n=s+s'$, $s\le s'$, $s, s'\in S$, respectively. Clearly,
$R_1(S,n)=R_2(S,n)+R_3(S,n)$ for all nonnegative integers $n$.

In \cite{nathanson78},
Nathanson studied the function $R_1(S,n)$ and called it the representation function. After that, these three functions were studied by Erd\H{o}s, S\'{a}rk\"{o}zy and S\'{o}s in  a series of papers (see \cite{erdos1}-\cite{erdos5}).
They showed that these functions have some different properties. In recent decades, representation functions have been extensively studied by
many authors [1,2,7,14,15,21,22,24,27]
and are still a fruitful area of research in additive number
theory.

For
$i\in \{1,2,3\},$ S\'{a}rk\"{o}zy asked whether there are sets
$A,B\subseteq \mathbb{N}$ with infinite symmetric difference such
that $R_i(A,n)=R_i(B,n)$ for all sufficiently large integers $n$.
In 2002, Dombi \cite{Dombi} proved that the answer is negative for $i=1$ and positive for $i=2$.
Chen and Wang \cite{ChenWang2003} proved that the answer is positive for $i=3$.
In fact, Dombi (for $i=2$), Chen and Wang (for $i=3$) proved that there exists a set $A\subseteq \mathbb{N}$
 such that $R_i(A,n)=R_i(\mathbb{N}\backslash A,n)$ for all $n\geq n_0.$ Lev \cite{Lev} gave a simple common proof to these
results. Later, S\'{a}ndor \cite{Sandor} gave a complete answer by
using generating functions and Tang \cite{Tang} gave
an elementary proof.


Let $[a,b]$ denote the set of all integers $n$ such that $a\le n\le b$. For  $a\in \mathbb{N}$ and  $S
\subseteq \mathbb{N}$,  $a+S$ denotes the set of all integers of
the form $a+s$, $s\in S$.
  Let $U$ be the set of all nonnegative integers which contains an even
  number of digits $1$ in their binary representations and $V=\mathbb{N}\setminus U$. Put $U_l=U\cap [0,2^l-1]$
and $V_l=V\cap [0,2^l-1]$. In Dombi's  result, the set $A$ is the set $U$ here.

In 2012, Yu and Tang \cite{YuTang} considered a special case where the intersection of two sets is non-empty, and proposed the following conjecture: Let $m\in \mathbb{N}$ and $R\subseteq\{0,1,\ldots,m-1\}$.
If $\mathbb{N}=A\cup B$ and $A\cap B=\{r+km:~k\in\mathbb{N},r\in R\}$, then $R_i(A,n)=R_i(B,n)$ cannot hold for all sufficiently large integers $n$, where $i=2,3$.
In 2016, Tang \cite{Tang16} confirmed this conjecture in the case that $R=\{0\}$.
In 2016, Chen and Lev \cite{ChenLev} disproved this conjecture for $i=2$.

In this paper, based on Chen and Lev's result, we focus on the function $R_2(S,n)$. For brevity, we denote
$R_2(S,n)$ by $R_S(n)$ from now on. Chen and Lev \cite{ChenLev} proved that for any given
positive integer $l$, there exist two subsets $A,B\subseteq
\mathbb{N}$ with $\mathbb{N}=A\cup B$ and $A\cap B=\{
2^{2l}-1+k(2^{2l+1}-1): k\in\mathbb{N}\}$ such that
$R_A(n)=R_B(n)$ for every positive integer $n$. The following two
problems are posed in \cite{ChenLev}.\vskip 2mm

\noindent{\bf Problem 1.} Given that $R_{C}(n)=R_{D}(n)$ for every
positive integer $n$, $C\cup D=\mathbb{N}$, and $C\cap
D=r+m\mathbb{N}$ with integers $r\geq 0$ and $m\geq 2$, must there
exist an integer $l\geq1$ such that $r=2^{2l}-1,
m=2^{2l+1}-1$?\vskip 2mm

\noindent{\bf Problem 2.} Given that $R_{C}(n)=R_{D}(n)$ for every
positive integer $n$, $C\cup D=[0,m]$, and $C\cap D=\{r\}$ with
integers $r\geq0$ and $m\geq2$, must there exist an integer
$l\geq1$ such that $r=2^{2l}-1, m=2^{2l+1}-2,
C=U_{2l}\cup(2^{2l}-1+V_{2l})$ and
$D=V_{2l}\cup(2^{2l}-1+U_{2l})$?\vskip 2mm

In 2017, Kiss and S\'{a}ndor
\cite{Kiss} solved Problem 2 affirmatively. In 2019, Chen, Tang and Yang
\cite{Chentangyang} solved Problem 1 affirmatively under the
assumption $0\le r<m$. In 2021, Chen and Chen \cite{ChenChen2021} solved Problem 1 affirmatively.

Notice that all of these results are taking into account the union of two sets is the set of all natural numbers or complete interval. In 2012, Jiao et al. \cite{Jiao} determined all sets $C$ and $D$ with $C\cup D=[0,m]\setminus \{r\}$, where $0<r<m$ and $C\cap D=\emptyset$
such that $R_C(n)=R_D(n)$ for any nonnegative integer $n$. In this paper, we consider the case
$C\cup D=\mathbb{N}\setminus\{r+mk:k\in\mathbb{N}\}$, where $m\ge 2$, $k\ge 0$ and $C\cap D=\emptyset$.


First, we give some notations. Let
$$H_0(h_1,h_2,\ldots)=\Big\{\sum_{i}\varepsilon_ih_i:\varepsilon_i\in\{0,1\},\sum_{i}\varepsilon_ih_i<+\infty,2\mid \sum_{i}\varepsilon_i\Big\}$$
and
$$H_1(h_1,h_2,\ldots)=\Big\{\sum_{i}\varepsilon_ih_i:\varepsilon_i\in\{0,1\},\sum_{i}\varepsilon_ih_i<+\infty,2\nmid \sum_{i}\varepsilon_i\Big\},$$
where $h_i\in \mathbb{Z}^{+}$.

For a nonnegative integer $l$, let
$$S_1(l)=H_0(h_1,h_2,\ldots),~~~T_1(l)= H_1(h_1,h_2,\ldots),$$
 where $h_1=1$, $h_2=2$, $h_3=4,\ldots,h_l=2^{l-1},h_{l+1}=2^l+1,h_{l+2}=2(2^l+1),h_{l+3}=4(2^l+1),\ldots$, and let
$$S_2(l)=H_0(h_1,h_2,\ldots),~~~T_2(l)= H_1(h_1,h_2,\ldots),$$
 where $h_1=1$, $h_2=2$, $h_3=4,\ldots,h_l=2^{l-1}+1,h_{l+1}=2^l+1,h_{l+2}=2(2^l+1),h_{l+3}=4(2^l+1),\ldots$

In this paper, we prove the following result.

\begin{thm}\label{thm2} Let $m\geq 2$ and $r\geq 0$ be two
integers and let $A$ and $B$ be two sets with $A\cup B=\mathbb{N}\setminus\{r+mk:k\in\mathbb{N}\}$ and $A\cap B=\emptyset$. Then $R_{A}(n)=R_{B}(n)$ for every positive integer $n$ if and only if
$A=S_1(l),~B=T_1(l)$
or
$A=S_2(l),~B=T_2(l)$
or
$A=S_1(l)+1,~B=T_1(l)+1$,
where $l$ is a nonnegative integer.
\end{thm}
\vskip 2mm

Throughout this paper, the characteristic function of the set $C$ is denoted by
\begin{eqnarray*}
\chi_{C}(t)=
\begin{cases}
0          &\mbox{ $t\not\in C$},\\
1&\mbox{ $t\in C$}.\\
\end{cases}
\end{eqnarray*}

Let $C(x)$ be the set of integers in $C$ which are less than or
equal to $x$. For $n\in \mathbb{N}$, let $R_{A,B}(n)$  denote the
number of the equation $n=a+b$, $a\in A, b\in B$. For any positive
integer $n$, let $D_2(n)$ denote the number of $1$ in the binary
representation of $n$. Let $D_2(0)=0$. It is clear that
$D_2(n+2^k)=1+D_2(n)$ if $0\le n<2^k$.

\section{Preliminary Lemmas}

\begin{lem}\label{999999} (See \cite[Claim 1]{Kiss}.) Let $0<r_1<\ldots<r_s\leq m$ be integers($s=0$ is allowed). Then there exists at most one pair of sets $(C,D)$ such that $C\cup D=[0,m]\setminus\{r_1,\ldots,r_s\}$, $C\cap D=\emptyset$, $0\in C$, and $R_C(k)=R_D(k)$ for every integer $k\leq m$.
\end{lem}

\begin{lem}\label{lem999999}(See \cite[Theorem 3]{Kiss}.) Let $C$ and $D$ be sets of nonnegative integers such that $C\cup D=[0,m]$, $C\cap D=\emptyset$ and $0\in C$. Then $R_C(n)=R_D(n)$ for every positive integer $n$ if and only if there exists an natural number $l$ such that $C=U_{l}$ and $D=V_{l}$.
\end{lem}
\begin{lem}\label{lem2n} Let $K$ and $L$ be two positive integers with $L\le K\leq2L$ and  let $A, B, C, D, T$ be five sets
such that $L\in T$, $L\not\in C$,  $0\in A\cap C$, $0\notin B\cup D$ and $$A\cup
B=\mathbb{N}\setminus T,\quad  A\cap B=\emptyset,$$ $$  C(K)\cup D(K)=[0,K]\setminus T(L-1),\quad C(K)\cap D(K)=\emptyset,$$
$$A(L-1)=C(L-1),\quad B(L-1)=D(L-1).$$
Then for any integer $n$ and integer $N$ with $L\leq n\leq N\leq K$, we have
\begin{eqnarray*}&&R_{A(n)}(N)+R_{D(n)}(N)-R_{B(n)}(N)-R_{C(n)}(N)\nonumber\\
&=&|D(n)\setminus (B(n)\cup T(L, n))|-R_{T,D(n)\setminus (B(n)\cup T(L, n))}(N)+R_{D,
T(L,n)\cap D }(N)\nonumber\\
&-&|C(n)\setminus (A(n)\cup T(L, n))|+R_{T,C(n)\setminus (A(n)\cup T(L, n))}(N) -R_{C,
T(L,n)\cap C }(N)-\varepsilon,
\end{eqnarray*}
where $T(L, n)=T\cap [L, n]$ and $\varepsilon=1$ if $N=2L$, otherwise $\varepsilon=0$.
\end{lem}
\begin{proof}
For any integer $n$ and integer $N$ with $L\leq n\leq N\leq K$,
if $N=a+b$ with $a,b\in A(n)$ and $a<b$, then by $2a<a+b=N\le K\le
2L$, we have $a<L$. It follows from $A(L-1)=C(L-1)$ and $a\in A(n)$
that $a\in C$ and so $a\in (C\cap A)(n)$. If $b\in C$, then $b\in
(C\cap A)(n)$. If $b\notin C$, then $b\in A(n)\setminus C(n)$. It
is clear that $$(C\cap A)(n) \cap (A(n)\setminus
C(n))=\emptyset,\quad (C\cap A)(n) \cup (A(n)\setminus
C(n))\subseteq A(n).$$ Thus,
\begin{eqnarray*}R_{A(n)}(N)= R_{(C\cap A)(n)
}(N)+R_{(C\cap A)(n), A(n)\setminus C(n) }(N).\end{eqnarray*} For
$a\in A(n)\setminus C(n)$, by $A(L-1)=C(L-1)$ and $L\not\in A$, we have $a>L$. Thus,
$N-a<2L-L=L\le n$. Again, by $A(L-1)=C(L-1)$ we have
$$N-a\in A\Leftrightarrow N-a\in C \Leftrightarrow N-a\in (C\cap
A)(n).$$ It follows that
$$R_{(C\cap A)(n), A(n)\setminus C(n) }(N)=R_{A, A(n)\setminus C(n) }(N)=R_{C, A(n)\setminus C(n)
}(N).$$ Hence
\begin{equation}\label{eqe2a}R_{A(n)}(N)=R_{(C\cap A)(n) }(N)+R_{C, A(n)\setminus
C(n) }(N) \end{equation}and
\begin{equation*}R_{A(n)}(N)=R_{(C\cap A)(n) }(N)+R_{A, A(n)\setminus
C(n) }(N) . \end{equation*} Similarly,
\begin{eqnarray}\label{eqe5}&&R_{D(n)}(N)\nonumber=R_{(D\cap B)(n) }(N)+R_{D(L-1), D(n)\setminus B(n)}(N)\nonumber\\
&=&R_{(D\cap B)(n) }(N)+R_{D, D(n)\setminus B(n)}(N)-\varepsilon\nonumber\\
&=&R_{(D\cap B)(n) }(N)+R_{D, D(n)\setminus (B(n)\cup(T(L,n))
}(N)+R_{D, T(L,n)\cap D(n) }(N)-\varepsilon,\end{eqnarray}
where $\varepsilon=1$ if $N=2L$, otherwise $\varepsilon=0$.

Now we prove that
\begin{equation}\label{eqe2}A(n)\setminus
C(n)=D(n)\setminus (B(n)\cup T(L, n)).\end{equation} In
fact, if $t\in A(n)\setminus C(n)$, then $t\in A(n)$, $t\notin
C(n)$, $t> L$ and $t\notin T(L,n)$. By $C(n)\cup D(n)=[0,n]\setminus T(L-1)$, we have $t\in D(n)$.
By $A\cap B=\emptyset$, we have $t\notin B(n)$.
 So $t\in D(n)\setminus (B(n)\cup T(L, n))$.  On the other hand, if $t\in D(n)\setminus (B(n)\cup T(L, n))$, then $t\in D(n)$, $t\notin B(n)$, $t\notin T(L,n)$. Noting that $D(L-1)=B(L-1)$ and $L\in T$, we have $t> L$ and $t\notin T(n)$. By $C(n)\cap D(n)=\emptyset$, we have $t\notin C(n)$.
By $A(n)\cup B(n)=[0,n]\setminus T(n)$ and $A\cap B=\emptyset$, we have $t\in A(n)$.
 So $t\in A(n)\setminus C(n)$.

Therefore,
\eqref{eqe2} holds. In view of \eqref{eqe2a},
\eqref{eqe2}, we have
\begin{equation}\label{eqe2ab}R_{A(n)}(N)=R_{(C\cap A)(n) }(N)
+R_{C, D(n)\setminus (B(n)\cup T(L, n))}(N).\end{equation}
For any $t\in D(n)\setminus (B(n)\cup T(L, n))$, by $D(L-1)=B(L-1)$ and $L\in T(L, n)$,
we have $t>L$. So $N-t\in [0, L-1]$.
 It follows from  $C(K)\cup D(K)=[0,K]\setminus T(L-1)$  that
\begin{eqnarray}\label{eqe4} &&R_{C, D(n)\setminus (B(n)\cup T(L, n))
}(N)+R_{D, D(n)\setminus (B(n)\cup T(L, n))}(N)\nonumber\\
&=&|D(n)\setminus (B(n)\cup T(L, n))|-R_{T(L-1),D(n)\setminus (B(n)\cup T(L, n))}(N)\nonumber\\
&=&|D(n)\setminus (B(n)\cup T(L, n))|-R_{T,D(n)\setminus (B(n)\cup T(L, n))}(N) .
\end{eqnarray}
For the last equality, we have used fact that $N-t\le L-1$ for all
$t\in D(n)\setminus (B(n)\cup T(L, n))$. By \eqref{eqe5}, \eqref{eqe2ab} and
\eqref{eqe4}, we have
\begin{eqnarray*}&& R_{A(n)}(N)+R_{D(n)}(N)\\
&=&R_{(C\cap A)(n) }(N)+R_{(D\cap B)(n) }(N)+|D(n)\setminus (B(n)\cup T(L, n))|\\
&& -R_{T,D(n)\setminus (B(n)\cup T(L, n))}(N)+R_{D, T(L,
n)\cap D }(N)-\varepsilon.\end{eqnarray*}
where $\varepsilon=1$ if $N=2L$, otherwise $\varepsilon=0$.

Noting that $L\not\in C$, similar to the proof above, we have
\begin{eqnarray*}&&R_{B(n)}(N)+R_{C(n)}(N)\\
&=&R_{(C\cap A)(n) }(N)+R_{(D\cap B)(n) }(N)+|C(n)\setminus (A(n)\cup T(L, n))|\\
&&-R_{T,C(n)\setminus (A(n)\cup T(L, n))}(N)+R_{C, T(L,
n)\cap C }(N).\end{eqnarray*}
 Thus,
\begin{eqnarray}\label{eqn2}&&R_{A(n)}(N)+R_{D(n)}(N)-R_{B(n)}(N)-R_{C(n)}(N)\\
&=&|D(n)\setminus (B(n)\cup T(L, n))|-R_{T,D(n)\setminus (B(n)\cup T(L, n))}(N)+R_{D,
T(L,n)\cap D }(N)\nonumber\\
&-&|C(n)\setminus (A(n)\cup T(L, n))|+R_{T,C(n)\setminus (A(n)\cup T(L, n))}(N) -R_{C,
T(L,n)\cap C }(N)-\varepsilon,\nonumber
\end{eqnarray}
where $\varepsilon=1$ if $N=2L$, otherwise $\varepsilon=0$.

This completes the proof of Lemma \ref{lem2n}.

\end{proof}

\begin{lem}\label{lem2} Let $K$ and $L$ be two positive integers with $L\le K\leq 2L$ and  let $A, B, C, D, T$ be five sets
such that $L\in T$,  $0\in A\cap C$, $0\notin B\cup D$ and $$A\cup
B=\mathbb{N}\setminus T,\quad  A\cap B=\emptyset,$$ $$  C(K)\cup D(K)=[0,K]\setminus T(L-1),\quad C(K)\cap D(K)=\emptyset.$$

 If $R_{A}(n)=R_{B}(n)$ and $R_{C}(n)=R_{D}(n)$ for every positive integer $n\le K$, then, $L\notin C$, $L\in D$ and for every positive integer $n<K$, we have
\begin{eqnarray*} &&R_{T,C\setminus (A\cup T(L, n))}(n)-R_{T,D\setminus (B\cup T(L, n))}(n)+R_{D, T(L, n)\cap
D }(n)-R_{C, T(L, n)\cap C }(n)\\
&=&R_{T,C\setminus (A\cup T(L, n))}(n+1)-R_{T,D\setminus (B\cup T(L, n))}(n+1)+R_{D, T(L,
n)\cap D }(n+1)\\
&&-R_{C, T(L, n)\cap C }(n+1)+\chi_A (n+1)-\chi_{C}(n+1)-\varepsilon,
\end{eqnarray*}
where $T(L, n)=T\cap [L, n]$ and $\varepsilon=1$ if $n=2L-1$, otherwise $\varepsilon=0$.
\end{lem}

\begin{proof} By the conditions, we have $0\in A\cap C$, $$A(L-1)\cup B(L-1)=[0, L-1]\setminus T(L-1)=C(L-1)\cup
D(L-1)$$ and $$A(L-1)\cap B(L-1)=\emptyset=C(L-1)\cap D(L-1).$$  Noting that
$0\notin B$, for every positive integer $n\le L-1$ we have
 $$R_{A(L-1)}(n)=R_{A}(n)=R_{B}(n)=R_{B(L-1)}(n)$$ and $$R_{C(L-1)}(n)=R_{C}(n)=R_{D}(n)=R_{D(L-1)}(n).$$   It follows from Lemma \ref{999999} that
 \begin{equation}\label{eqe1}A(L-1)=C(L-1),\quad B(L-1)=D(L-1).\end{equation}
 Since $0\in C$, it follows from \eqref{eqe1} that
\begin{eqnarray*} 0&=&R_C(L)-R_D(L)\\
&=&R_{C(L)}(L)-R_{D(L)}(L)\\
&=&R_{C(L-1)}(L)+\chi_C(L)-R_{D(L-1)}(L)\\
&=&R_{A(L-1)}(L)+\chi_C(L)-R_{B(L-1)}(L)\\
&=&R_{A(L)}(L)+\chi_C(L)-R_{B(L)}(L)\\
&=&\chi_C(L).
\end{eqnarray*}
So $L\notin C$. By $C(K)\cup D(K)=[0,K]\setminus T(L-1)$, we have
$L\in D$. For $n< L$, we have $T(L, n)=\emptyset$ since $[L,
n]=\emptyset $ and by \eqref{eqe1}, we have
$$D(n)\setminus B(n)=\emptyset,\quad C(n)\setminus A(n)=\emptyset,
\quad \chi_A (n+1)=\chi_{C}(n+1).$$ Noting that $0\in A$, we
have $0\notin T$. Thus, Lemma \ref{lem2} is true trivially for all
$n<L$.
 In the following, we assume that $L\le n<K$. Let $N$ be an integer with $n\le N\le K$.
Noting that
$$R_{A(n)}(n)=R_A (n)=R_B(n)=R_{B(n)}(n)$$ and $$
R_{D(n)}(n)=R_{D}(n)=R_{C}(n)=R_{C(n)}(n),$$ by Lemma \ref{lem2n} with $N=n$, we
have
\begin{eqnarray}\label{eqn1}&&|C(n)\setminus (A(n)\cup T(L, n))|-|D(n)\setminus (B(n)\cup T(L, n))|\nonumber\\
&=&R_{T,C(n)\setminus (A(n)\cup T(L, n))}(n)-R_{T,D(n)\setminus (B(n)\cup T(L, n))}(n)+R_{D,
T(L,n)\cap D }(n)-R_{C, T(L,n)\cap C }(n\nonumber)\\
&=&R_{T,C\setminus (A\cup T(L,n))}(n)-R_{T,D\setminus (B\cup T(L,n))}(n)+R_{D, T(L,n)\cap D
}(n)-R_{C, T(L,n)\cap C }(n) .\end{eqnarray} Noting that $0\in
A\cap C$(so $0\notin T$) and $0\notin B\cup D$, we have
$$R_{A(n)}(n+1)=R_A (n+1)-\chi_A (n+1)=R_B(n+1)-\chi_A (n+1)=R_{B(n)}(n+1)-\chi_A (n+1),$$ $$
R_{D(n)}(n+1)=R_{D}(n+1)
=R_{C}(n+1)=R_{C(n)}(n+1)+\chi_{C}(n+1),$$
\begin{eqnarray*}&&R_{T,D(n)\setminus (B(n)\cup T(L,n))}(n+1)-R_{T,C(n)\setminus
(A(n)\cup T(L,n))}(n+1)\\
&=&R_{T,D\setminus (B\cup T(L,n))}(n+1)-R_{T,C\setminus (A\cup T(L,n))}(n+1).
\end{eqnarray*}
Thus, by Lemma \ref{lem2n} with $N=n+1$, we have
\begin{eqnarray}\label{eqe8}&&|C(n)\setminus (A(n)\cup T(L, n))|-|D(n)\setminus (B(n)\cup T(L, n))|\nonumber\\
&=&R_{T,C\setminus (A\cup T(L,n))}(n+1)-R_{T,D\setminus (B\cup T(L,n))}(n+1)+R_{D, T(L,n)\cap D
}(n+1)\nonumber\\
&&-R_{C, T(L,n)\cap C }(n+1)+\chi_A (n+1)-\chi_{C}(n+1)-\varepsilon,\end{eqnarray}
where $\varepsilon=1$ if $n=2L-1$, otherwise $\varepsilon=0$.

Now Lemma \ref{lem2} follows from \eqref{eqn1} and \eqref{eqe8} immediately.
\end{proof}

\begin{lem}\label{lem3} Let $A, B$ and $T$ be three sets such that $A\cup B=\mathbb{N}\setminus T$, $A\cap B=\emptyset$, $0\in A$ and $0\notin T$.
 If $R_{A}(n)=R_{B}(n)$ for every positive integer $n$, then $\min T\notin
U$, $\min T\in V$ and for every positive integer $n<2\min
 T$,
\begin{eqnarray*} \sum_{t\in T(n)}  \chi_{U}
(n-t+1)
 =\sum_{t\in T(n)} \chi_{U} (n-t)+\chi_A
(n+1)-\chi_{U}(n+1)-\varepsilon,
\end{eqnarray*}
where $\varepsilon=1$ if $n=2L-1$, otherwise $\varepsilon=0$.
\end{lem}

\begin{proof} We employ Lemma \ref{lem2} with $C=U$, $D=V$, $L=\min
T$ and $K=2\min T$. By Lemma \ref{lem2}, we have $\min T\notin U$ and
$\min T\in V$. By \eqref{eqe1}, we have
$$\min T+\min (D\setminus (B\cup T(L,n)))>L+L=2L,~~ \min T+\min (C\setminus
(A\cup T(L,n)))>L+L=2L.$$ Hence, for all $n< 2\min T$,
$$R_{T,D\setminus (B\cup T(L,n))}(n)-R_{T,C\setminus
(A\cup T(L,n))}(n)=0,$$$$R_{T,D\setminus (B\cup T(L,n))}(n+1)-R_{T,C\setminus
(A\cup T(L,n))}(n+1)=0.$$
 Noting that  $T(L, n)=T\cap [L, n]=T\cap
[0, n]=T(n)$ and $ \chi_{U} (k)+ \chi_{V} (k)=1$ for all integers
$k\ge 0$, we have
\begin{eqnarray*}&& R_{V, T(L,n)\cap
V }(n)-R_{U, T(L,n)\cap U }(n)\\
 &=&  R_{V, T(n)\cap
V }(n)-R_{U, T(n)\cap U }(n)\\
 &=& \sum_{t\in T(n)} \chi_{V} (n-t)\chi_{V}
 (t)-\sum_{t\in T(n)} \chi_{U} (n-t)\chi_{U}
 (t)\\
&=& \sum_{t\in T(n)} (1-\chi_{U} (n-t))(1-\chi_{U}
 (t))-\sum_{t\in T(n)} \chi_{U} (n-t) \chi_{U}
 (t)\\
&=& \sum_{t\in T(n)} (1-(\left( \chi_{U} (n-t)+\chi_{U}
 (t)\right)).
\end{eqnarray*}
Similarly, we have $$R_{V, T(L,n)\cap V }(n+1)-R_{U, T(L,n)\cap U
}(n+1)=\sum_{t\in T(n)}(1-(\left( \chi_{U} (n-t+1)+\chi_{U}
 (t)\right)).$$
It follows from Lemma \ref{lem2}   that
\begin{eqnarray*}&&\sum_{t\in T(n)}\chi_{U} (n-t+1)\\
 &=&\sum_{t\in T(n)}\chi_{U} (n-t)+\chi_A (n+1)-\chi_{U}(n+1)-\varepsilon,
\end{eqnarray*}
where $\varepsilon=1$ if $n=2L-1$, otherwise $\varepsilon=0$.

Now Lemma \ref{lem3} follows immediately.
\end{proof}

\begin{lem}\label{lemy1} (See \cite[Theorem 3]{Kiss1}.)  Let
$A=H_0(h_1,h_2,\ldots)$, $B=H_1(h_1,h_2,\ldots),$
where $h_i\in \mathbb{Z}^{+}$.
Then $R_{A}(n)=R_{B}(n)$ for every positive integer $n$.
\end{lem}

\begin{lem}\label{lem1w} Let $u$ be an nonnegative integer and
 let $E_u$, $F_u$ be sets of nonnegative integers such that $E_u\cup F_u=[0,2^{u+1}+1+2^u]\setminus\{2^u\}$, $0\in E_u$ and $E_u\cap F_u=\emptyset$.
 Then $R_{E_u}(n)=R_{F_u}(n)$ for every positive integer $1\leq n\leq 2^{u+1}+1+2^u$ if and only if
$$E_u=U_{u}\cup(2^{u}+1+V_{u})\cup(2^{u+1}+1+V_{u}),$$
$$F_u=V_{u}\cup(2^{u}+1+U_{u})\cup(2^{u+1}+1+U_{u})\cup\{2^{u+1}+1+2^u\}.$$
\end{lem}
\begin{proof} Sufficiency. It is clear that
$$E_u\cup F_u=[0,2^{u+1}+1+2^u]\setminus\{2^u\},~~0\in E_u,~~ E_u\cap F_u=\emptyset.$$

If $1\leq n\leq 2^{u+1}+1+2^u$, then
$$R_{E_u}(n)=R_{U_{u}}(n)+R_{U_{u},V_{n}}(n-(2^{u}+1))+R_{U_{u},V_{n}}(n-(2^{u+1}+1))+R_{V_{u}}(n-(2^{u+1}+2)),$$
$$R_{F_u}(n)=R_{V_{u}}(n)+R_{U_{u},V_{n}}(n-(2^{u}+1))+R_{U_{u},V_{n}}(n-(2^{u+1}+1))+R_{U_{u}}(n-(2^{u+1}+2)),$$
it follows from Lemma \ref{lem999999} that $R_E(n)=R_F(n)$.

Necessity. The necessity follows from Lemma \ref{999999} and the sufficiency.

This completes the proof of Lemma \ref{lem1w}.
\end{proof}

\begin{lem}\label{lemh}Let $u$ and $M$ be two nonnegative integers such that $2^{u+1}+2\leq M\leq2^{u+1}+2^u$ and
 let $E_u$, $F_u$ be defined as Lemma \ref{lem1w}. Let $A, B$ and $T$ be three sets such that $A\cup B=\mathbb{N}\setminus T$, $T=\{2^u<M<\ldots\}$, $\min (T\setminus \{2^u,M\})\geq 2^{u+1}+2^u+1$, $A\cap B=\emptyset$, $0\in A$.
 If $R_{A}(n)=R_{B}(n)$ for every positive integer $n$, then $M\notin
E_u$, $M\in F_u$ and for every positive integer $n$ with $M\leq n<2^{u+1}+2^{u}+1$, we have
\begin{equation}\label{bx1}\chi_{F_u}(n-M)=\chi_{F_u}(n-M+1)+\chi_A(n+1)-\chi_{E_u}(n+1).\end{equation}
\end{lem}
\begin{proof}Now we employ Lemma \ref{lem2} with $C=E_u$, $D=F_u$ and $L=M$, where $E_u$ and $F_u$ are defined as in Lemma \ref{lem1w}. Then $M\notin
E_u$ and $M\in F_u$. For any $L\leq n< 2^{u+1}+1+2^u$, by \eqref{eqe1} and $L\in T(L,n)$, we have $\min (D\setminus (B\cup T(L, n)))>
L$ and $\min (C\setminus (A\cup T(L, n)))>L$. Thus $$\min T+\min (C\setminus (A\cup T(L, n)))>
2^u+L=2^u+M\geq 2^{u+1}+2^u+2,$$ $$ \min T+\min (D\setminus (B\cup T(L, n)))> 2^u+M\geq 2^{u+1}+2^u+2.$$ Hence
$R_{T,C\setminus (A\cup T(L, n))}(n)-R_{T,D\setminus (B\cup T(L, n))}(n)=0$ for $n\le 2^{u+1}+2^u+2$.
Thus, for $M\leq n< 2^{u+1}+2^u+1$, we have
\begin{eqnarray*} &&R_{F_u, T(L, n)\cap
F_u }(n)-R_{E_u, T(L, n)\cap E_u }(n) \\
&=&R_{F_u, T(L, n)\cap
F_u }(n+1)-R_{E_u, T(L, n)\cap E_u }(n+1)  +\chi_A
(n+1)-\chi_{E_u}(n+1),
\end{eqnarray*}
Noting that $M\in F_u$ and $\min (T\setminus \{2^u,M\})\geq 2^{u+1}+2^u+1$, we have
$$\chi_{F_u}(n-M)=\chi_{F_u}(n-M+1)+\chi_A(n+1)-\chi_{E_u}(n+1).$$

This completes the proof of Lemma \ref{lemh}.
\end{proof}

\begin{lem}\label{lemyc2}  Let $h_1=2$, $h_2=3$, $h_3=4,\ldots,h_l=2^{l-1},\ldots$. Let
$$X=H_0(h_1,h_2,\ldots),~~~Y=H_1(h_1,h_2,\ldots).$$
Then
$$X\cup Y=\mathbb{N}\setminus\{1\},~~X\cap Y=\emptyset$$
and
$R_{X}(n)=R_{Y}(n)$ for every positive integer $n$.
\end{lem}
\begin{proof} It is clear that
$$X\cup Y=\mathbb{N}\setminus\{1\},~~X\cap Y=\emptyset.$$
Then Lemma \ref{lemyc2} follows from Lemma \ref{lemy1} immediately.
\end{proof}

\begin{lem}\label{chen55555}Let $m$ be a positive integer with $m\geq2$. Let $A$ and $B$ be sets such that $A\cup B=\mathbb{N}\setminus \{1+mk:k\in \mathbb{N}\}$, $A\cap B=\emptyset$ and $R_A(n)=R_B(n)$ for every positive integer $n$. Then $m=2$ or $3$.
\end{lem}
\begin{proof}Now we employ Lemma \ref{lem2} with $C=X$, $D=Y$, $L=1+m$ and $K=2L$, where $X$ and $Y$ are defined as in Lemma \ref{lemyc2}. Then $L\not\in X$ and $L\in Y$. Assume that $m\geq4$. Then by $L=m+1\in Y$, we have $m\geq5$. For any $L\leq n< K$, by \eqref{eqe1} and $L\in T(L,n)$, we have $\min (X\setminus (B\cup T(L, n)))>
L$ and $\min (Y\setminus (A\cup T(L, n)))>L$. Thus $$\min T+\min (X\setminus (A\cup T(L, n)))>
1+L,$$ $$ \min T+\min (Y\setminus (B\cup T(L, n)))>1+L.$$ Hence
$R_{T,X\setminus (A\cup T(L, n))}(n)-R_{T,Y\setminus (B\cup T(L, n))}(n)=0$ for $n\le 1+L$. Noting that $m\geq 5$, we have
$$L+5\leq L+m=1+2m\leq2L-1.$$
 By Lemma \ref{lem2} with $n=L$, we have
\begin{eqnarray*} &&R_{Y, T(L, n)\cap Y }(n)-R_{X, T(L, n)\cap X }(n)\\
&=&R_{Y, T(L,n)\cap Y }(n+1)-R_{X, T(L, n)\cap X }(n+1)+\chi_A (n+1)-\chi_{X}(n+1).
\end{eqnarray*}
Noting that $L\not\in X$, $L\in Y$ and $L<1+2m$, we have
\begin{eqnarray*} \chi_{Y}(0)=\chi_{Y}(1)+\chi_A (L+1)-\chi_{X}(L+1).
\end{eqnarray*}
Noting that $\chi_{Y}(0)=\chi_{Y}(1)=0$, we have $\chi_A (L+1)=\chi_{X}(L+1)$.
Hence,
$$R_{T,X\setminus (A\cup T(L, L+1))}(L+2)-R_{T,Y\setminus (B\cup T(L, L+1))}(L+2)=0.$$
By Lemma \ref{lem2} with $n=L+1$, we have
\begin{eqnarray*} \chi_{Y}(1)=\chi_{Y}(2)+\chi_A (L+2)-\chi_{X}(L+2).
\end{eqnarray*}
Noting that $\chi_{Y}(1)=0,~\chi_{Y}(2)=1$, we have $\chi_A (L+2)=0,~\chi_{X}(L+2)=1$. Hence,
$$R_{T,X\setminus (A\cup T(L, L+2))}(L+2)-R_{T,Y\setminus (B\cup T(L, L+2))}(L+2)=0,$$
$$R_{T,X\setminus (A\cup T(L, L+2))}(L+3)-R_{T,Y\setminus (B\cup T(L, L+2))}(L+3)=1.$$
By Lemma \ref{lem2} with $n=L+2$, we have
\begin{eqnarray*} \chi_{Y}(2)=1+\chi_{Y}(3)+\chi_A (L+3)-\chi_{X}(L+3).
\end{eqnarray*}
Noting that $\chi_{Y}(2)=1,~\chi_{Y}(3)=1$, we have $\chi_A (L+3)=0,~\chi_{X}(L+3)=1$. Hence,
$$R_{T,X\setminus (A\cup T(L, L+3))}(L+3)-R_{T,Y\setminus (B\cup T(L, L+3))}(L+3)=1,$$
$$R_{T,X\setminus (A\cup T(L, L+3))}(L+4)-R_{T,Y\setminus (B\cup T(L, L+3))}(L+4)=1.$$
By Lemma \ref{lem2} with $n=L+3$, we have
\begin{eqnarray*} \chi_{Y}(3)=\chi_{Y}(4)+\chi_A (L+4)-\chi_{X}(L+4).
\end{eqnarray*}
Noting that $\chi_{Y}(3)=1,~\chi_{Y}(4)=1$, we have $\chi_A (L+4)=\chi_{X}(L+4)$. Hence,
$$R_{T,X\setminus (A\cup T(L, L+4))}(L+4)-R_{T,Y\setminus (B\cup T(L, L+4))}(L+4)=1,$$
$$R_{T,X\setminus (A\cup T(L, L+4))}(L+5)-R_{T,Y\setminus (B\cup T(L, L+4))}(L+5)=0.$$
By Lemma \ref{lem2} with $n=L+4$, we have
\begin{eqnarray*} 1+\chi_{Y}(4)=\chi_{Y}(5)+\chi_A (L+5)-\chi_{X}(L+5).
\end{eqnarray*}
Noting that $\chi_{Y}(4)=1,~\chi_{Y}(5)=0$, we have $\chi_A (L+2)-\chi_{X}(L+2)=2$, a contradiction. Hence, $m=2$ or $3$.

This completes the proof of Lemma \ref{chen55555}.
\end{proof}

\begin{lem}\label{55555}Let $m$ and $r$ be two integers with $m\geq 2$, $r\geq 0$ and $m\geq r$. Let $A$ and $B$ be sets such that $A\cup B=\mathbb{N}\setminus \{r+mk:k\in \mathbb{N}\}$, $A\cap B=\emptyset$ and $R_A(n)=R_B(n)$ for every positive integer $n$. Then there exist a nonnegative integer $v$ such that $r=0$ and $m=2^{v}+1$ or there exist nonnegative integer $u$ and $\varepsilon\in\{0,1\}$ such that $r=2^u$ and $m=2^{u+\varepsilon}+1$.
\end{lem}

\begin{proof}  Firstly, we discuss the case of $r\geq1$. If $r=1$, then the result is true follows from Lemma \ref{chen55555}. We may assume that $r\geq2$. Let $u$ be the integer with $2^u\leq r<2^{u+1}$. Then $u\geq1$. Firstly, we will prove $r=2^u$.
Suppose that $2^u< r<2^{u+1}$. We will derive a contradiction.

{\bf{Case 1.}} $2\nmid r$. Let $$r=\sum\limits_{i=0}^{j}2^i+2^{j+2}R$$
with $j,R\in\mathbb{N}$. Noting that $r\geq3$, we have
$$r<r+1\leq r+r-2<2r-1.$$
By Lemma \ref{lem3} with $n=r$, we have

\begin{equation}\label{eq31}\chi_{U}(1)=\chi_{U}(0)+\chi_{A}(r+1)-\chi_{U}(r+1).\end{equation}
 Since $$\chi_{U}(1)=0,~~~\chi_{U}(0)=1,$$
 it follows from \eqref{eq31} that $\chi_{U}(r+1)=1$. Noting that $r\notin U$, we have $2\nmid j$ and $2\nmid D_2(R)$, and so $R\geq1$. Thus, $2^u+\sum\limits_{i=0}^{1}2^i\leq r$, which implies that $2^u\leq r-3$. So $r+2^u<2r-1$. Noting that $2^u<r\leq m$ and, we have $r+2^u<r+m$.
By Lemma \ref{lem3} with $n=r+2^u$, we have
 \begin{equation}\label{eq33}\chi_{U}(2^u+1)=\chi_{U}(2^u)+\chi_{A}(r+2^u+1)-\chi_{U}(r+2^u+1).\end{equation}
 Noting that $D_2(r+2^u+1)=1+D_2(R)$ is even, we have $r+2^u+1\in U$.
Since $$\chi_{U}(2^u+1)=1,~~~\chi_{U}(2^u)=0,$$
 it follows from \eqref{eq33} that $\chi_{A}(r+2^u+1)=2$, a contradiction.

{\bf{Case 2.}} $2\mid r$.  Noting that $2^u+2\leq r$, we have $r+2^u\leq 2r-2$. By Lemma \ref{lem3} with $n=r+2^u<r+m$, we have

\begin{equation}\label{ceq34}\chi_{U}(2^u+1)=\chi_{U}(2^u)+\chi_{A}(r+2^u+1)-\chi_{U}(r+2^u+1).\end{equation}
Since $\chi_{U}(2^u+1)=1,~\chi_{U}(2^u)=0$ and $\chi_{U}(r+2^u+1)=1$ from $2\mid r$,
 it follows from \eqref{ceq34} that $\chi_{A}(r+2^u+1)=2$, a contradiction. Hence, $r=2^u$.

Next we will prove that $m=2^{u}+1$ or $2^{u+1}+1$.
Assume that $m\geq2^{u+1}+2$. It follows from Lemma \ref{999999} and Lemma \ref{lem1w} that
$$A(2^{u+1}+2^u+1)=E,~~B(2^{u+1}+2^u+1)=F.$$
By Lemma \ref{lem1w}, we have $0\in A$ and $1,2^{u+1}+1+2^u\in F$.
Hence,
\begin{eqnarray*}&&R_A(2^{u+1}+2^u+2)\\
&=&R_{A(2^{u+1}+2^u+2)}(2^{u+1}+2^u+2)\\
&=&R_{A(2^{u+1}+2^u+1)}(2^{u+1}+2^u+2)+\chi_A(2^{u+1}+2^u+2)\\
&=&R_{E_u}(2^{u+1}+2^u+2)+\chi_A(2^{u+1}+2^u+2)\\
&=&R_{E_u(2^{u+1}+2^u)}(2^{u+1}+2^u+2)+\chi_A(2^{u+1}+2^u+2)\\
&=&R_{V_u}(2^u)+R_{U_u,V_u}(2^u+1)+\chi_A(2^{u+1}+2^u+2),
\end{eqnarray*}
\begin{eqnarray*}&&R_B(2^{u+1}+2^u+2)\\
&=&R_{B(2^{u+1}+2^u+2)}(2^{u+1}+2^u+2)\\
&=&R_{B(2^{u+1}+2^u+1)}(2^{u+1}+2^u+2)\\
&=&R_{F_u}(2^{u+1}+2^u+2)\\
&=&R_{F_u(2^{u+1}+2^u)}(2^{u+1}+2^u+2)+1\\
&=&R_{U_u}(2^u)+R_{U_u,V_u}(2^u+1)+2,
\end{eqnarray*}
it follows from Lemma \ref{lem999999} that $\chi_A(2^{u+1}+2^u+2)=2$, a contradiction.
Hence, $m\leq 2^{u+1}+1$. If $m=r=2^u$, then by Lemma \ref{lem3} with $n=2r-1=2^{u+1}-1< r+m$, we have

\begin{equation}\label{eq34}\chi_{U}(2^u)=\chi_{U}(2^u-1)+\chi_{A}(2^{u+1})-\chi_{U}(2^{u+1})-1.\end{equation}
Since $$\chi_{U}(2^u)=\chi_{U}(2^{u+1})=\chi_{A}(2^{u+1})=0,$$
it follows from \eqref{eq34} that $\chi_{U}(2^u-1)=1$. So $2\mid u$.
For any integer $0\leq n\leq 2^{u+1}-1$, we have
$$R_{A(2^{u+1}-1)}(n)=R_{B(2^{u+1}-1)}(n),$$
$$R_{E_u(2^{u+1}-1)}(n)=R_{F_u(2^{u+1}-1)}(n),$$
$$A(2^{u+1}-1)\cap B(2^{u+1}-1)=E_u(2^{u+1}-1)\cap F_u(2^{u+1}-1)=\emptyset$$
and
$$A(2^{u+1}-1)\cup B(2^{u+1}-1)=[0,2^{u+1}-1]\setminus\{2^u\}=E_u(2^{u+1}-1)\cup F_u(2^{u+1}-1).$$
It follows from Lemma \ref{999999} that
$$A(2^{u+1}-1)=E_u(2^{u+1}-1),~~ B(2^{u+1}-1)= F_u(2^{u+1}-1).$$
Noting that $0\in A$, $r+m=2^{u+1}\not\in A\cup B$ and $2\mid u$, we have
\begin{eqnarray*}R_{A}(2^{u+1}+1)&=&R_{A(2^{u+1}-1)}(2^{u+1}+1)+\chi_{A}(2^{u+1}+1)\\
&=&R_{E_u(2^{u+1}-1)}(2^{u+1}+1)+\chi_{A}(2^{u+1}+1)\\
&=&R_{E_u(2^{u+1})}(2^{u+1}+1)+\chi_{A}(2^{u+1}+1)\\
&=&R_{U_u,V_u}(2^u)+\chi_{A}(2^{u+1}+1),
\end{eqnarray*}
\begin{eqnarray*}&&R_{B}(2^{u+1}+1)=R_{B(2^{u+1}-1)}(2^{u+1}+1)=R_{F_u(2^{u+1}-1)}(2^{u+1}+1)\\
\\
&=&R_{F_u\setminus\{2^{u+1}\}}(2^{u+1}+1)=R_{F_u(2^{u+1})}(2^{u+1}+1)-1=R_{U_u,V_u}(2^u)-1,
\end{eqnarray*}
It follows that $\chi_{A}(2^{u+1}+1)=-1$, a contradiction. Hence, $2^{u}+1\leq m\leq 2^{u+1}+1$.

Suppose that $2^{u}+2\leq m\leq 2^{u+1}$. We will derive a contradiction.
Let $M=r+m$. Then $2^{u+1}+2\leq M\leq 2^{u+1}+2^u$.

{\bf Case 1.} $M$ is odd. Let
$$M=2^{u+1}+1+\sum\limits_{i=c}^{t}2^i+2^{t+2}R$$
with $t\geq c\geq1$ and $R\in \mathbb{N}$. If $R\geq 1$, then by $M\in F_u$, we have $\chi_{E_u}(M+2^t+1)=1$. If $R=0$ and $t<u-1$, then by $M\in F_u$, we have $\chi_{E_u}(M+2^t+1)=1$. Noting that $m\geq2^{u}+2$ and $M\geq2^{u+1}+2$ , we have
$$r+2m=M+m\geq2^{u+1}+2^u+4.$$
By \eqref{bx1} with $n=M+2^t<2^{u+1}+2^u$, we have
\begin{equation*}\chi_{F_u}(2^t)=\chi_{F_u}(2^t+1)+\chi_A(M+2^t+1)-\chi_{E_u}(M+2^t+1).\end{equation*}
It follows from $\chi_{F_u}(2^t)=1$ and $\chi_{F_u}(2^t+1)=0$ that $\chi_A(M+2^t+1)=2$, a contradiction.
 So $R=0$ and $t=u-1$. Hence,
$$M=2^{u+1}+1+\sum\limits_{i=c}^{u-1}2^i.$$
{\bf Case 2.} $M$ is even. Let
$$M=2^{u+1}+1+\sum\limits_{i=0}^{t}2^i+2^{t+2}R$$
with $t\geq0$ and $R\in \mathbb{N}$. Noting that $M\in F_u$, we have $M\geq2^{u+1}+1+3$. Let $k$ be a positive integer with $2^k\leq M-(2^{u+1}+1)<2^{k+1}$. Then $k\leq u-1$. If $k<u-1$,
then by \eqref{bx1} with $n=M+2^t$, we have
\begin{equation*}\chi_{F_u}(2^t)=\chi_{F_u}(2^t+1)+\chi_A(M+2^t+1)-\chi_{E_u}(M+2^t+1),\end{equation*}
it follows from $\chi_{F_u}(2^t)=1$ and $\chi_{F_u}(2^t+1)=0$ that $\chi_{E_u}(M+2^t+1)=0$. So $t$ is odd and $2\mid D_2(R)$.
By \eqref{bx1} with $n=M+2^{t+1}+2^{t+2}R$, we have
\begin{equation*}\chi_{F_u}(2^{t+1}+2^{t+2}R)=\chi_{F_u}(2^{t+1}+2^{t+2}R+1)+\chi_A(M+2^{t+1}+2^{t+2}R+1)-\chi_{E_u}(M+2^{t+1}+2^{t+2}R+1),\end{equation*}
it follows from $$\chi_{F_u}(2^{t+1}+2^{t+2}R)=1,\chi_{F_u}(2^{t+1}+2^{t+2}R+1)=0$$ and $\chi_{E_u}(M+2^{t+1}+2^{t+2}R+1)=1$ that
$\chi_A(M+2^{t+1}+2^{t+2}R+1)=2$, a contradiction.
So $k=u-1$.
By \eqref{bx1} with $n=M$, we have
\begin{equation*}\chi_{F_u}(0)=\chi_{F_u}(1)+\chi_A(M+1)-\chi_{E_u}(M+1),\end{equation*}
it follows from $\chi_{F_u}(0)=0,\chi_{F_u}(1)=1$ that $\chi_{E_u}(M+1)=1$. Noting that $M\notin E_u$, we have $2\nmid t$ and $2\mid D_2(R)$. If $R=0$, then
$$M=2^{u+1}+1+\sum\limits_{i=0}^{u-1}2^i,$$
and so $\chi_{E_u}(M+1)=0$, a contradiction.
If $R\geq1$, then
$$M=2^{u+1}+1+\sum\limits_{i=0}^{t}2^i+\sum\limits_{i=t'}^{h}2^i+2^{h+2}R'$$
with $h\geq t'\geq t+2$ and $R'\in\mathbb{N}$.
If $R'\geq 1$, then by \eqref{bx1} with $n=M+2^{h}$, we have
\begin{equation*}\chi_{F_u}(2^{h})=\chi_{F_u}(2^{h}+1)+\chi_A(M+2^{h}+1)-\chi_{E_u}(M+2^{h}+1),\end{equation*}
it follows from $\chi_{F_u}(2^{h})=1,\chi_{F_u}(2^{h}+1)=0$ and $\chi_{E_u}(M+2^{h}+1)=1$ that
$\chi_A(M+2^{h}+1)=2$, a contradiction. So $R'=0$. Hence,
$$M=2^{u+1}+1+\sum\limits_{i=0}^{t}2^i+\sum\limits_{i=t'}^{u-1}2^i$$
with $t'\geq t+2$.
 If $t'\geq t+3$,  then by \eqref{bx1} with $n=M+2^{t+1}$, we have
\begin{equation*}\chi_{F_u}(2^{t+1})=\chi_{F_u}(2^{t+1}+1)+\chi_A(M+2^{t+1}+1)-\chi_{E_u}(M+2^{t+1}+1),\end{equation*}
it follows from $\chi_{F_u}(2^{t+1})=1,\chi_{F_u}(2^{t+1}+1)=0$ and $\chi_{E_u}(M+2^{t+1}+1)=1$ that
$\chi_A(M+2^{t+1}+1)=2$, a contradiction. Hence,
$$M=2^{u+1}+1+\sum\limits_{i=0}^{t}2^i+\sum\limits_{i=t+2}^{u-1}2^i.$$
By Lemma \ref{lem2n} with $C(2^{u+1}+1+2^u)=E_u$, $D(2^{u+1}+1+2^u)=F_u$ and $L=M$, $K=2^{u+1}+2^u+2$, where $E_u$ and $F_u$ are as in Lemma \ref{lem1w}.
For any $L\leq n\leq 2^{u+1}+1+2^u$, by \eqref{eqe1} and $L\in T(L,n)$, we have $\min (D\setminus (B\cup T(L, n)))>
L$ and $\min (C(n)\setminus (A\cup T(L, n)))>L$. Thus $$\min T+\min (C\setminus (A\cup T(L, n)))>
2^u+L=2^u+M\geq2^{u+1}+2+2^u,$$ $$ \min T+\min (D\setminus (B\cup T(L, n)))> 2^u+L=2^u+M\geq2^{u+1}+2+2^u.$$ Hence
$R_{T,C(n)\setminus (A\cup T(L, n))}(n)-R_{T,D\setminus (B\cup T(L, n))}(n)=0$ for $n\le 2^{u+1}+2+2^u$.
By Lemma \ref{lem2n} with $n=2^{u+1}+2^u+1$ and $N=2^{u+1}+2^u+1$, noting that
$$R_{E_u}(2^{u+1}+2^u+1)=R_{F_u}(2^{u+1}+2^u+1)$$
and
$$R_{A(2^{u+1}+2^u+1)}(2^{u+1}+2^u+1)=R_{B(2^{u+1}+2^u+1)}(2^{u+1}+2^u+1),$$
we have
\begin{eqnarray}\label{k1}0&=&|F_u\setminus (B(2^{u+1}+2^u+1)\cup T(M, 2^{u+1}+2^u+1))|\\
&&-|E_u\setminus (A(2^{u+1}+2^u+1)\cup T(M, 2^{u+1}+2^u+1))|+\chi_{F_u}(2^{u+1}+2^u+1-M).\nonumber
\end{eqnarray}
By Lemma \ref{lem2n} with $n=2^{u+1}+2^u+1$ and $N=2^{u+1}+2^u+2$, noting that
\begin{eqnarray*}
&&R_{B(2^{u+1}+2^u+1)}(2^{u+1}+2^u+2)\\
&=&R_{B(2^{u+1}+2^u+2)}(2^{u+1}+2^u+2)\\
&=&R_{A(2^{u+1}+2^u+2)}(2^{u+1}+2^u+2)\\
&=&R_{A(2^{u+1}+2^u+1)}(2^{u+1}+2^u+2)+\chi_A(2^{u+1}+2^u+2),
\end{eqnarray*}
we have
\begin{eqnarray*}&&R_{F_u}(2^{u+1}+2^u+2)-R_{E_u}(2^{u+1}+2^u+2)\nonumber\\
&=&|F_u\setminus (B(2^{u+1}+2^u+1)\cup T(M, 2^{u+1}+2^u+1))|\nonumber\\
&&-|E_u\setminus (A(2^{u+1}+2^u+1)\cup T(M, 2^{u+1}+2^u+1))|\nonumber\\
&&+\chi_{F_u}(2^{u+1}+2^u+2-M)+\chi_A(2^{u+1}+2^u+2).
\end{eqnarray*}
it follows from \eqref{k1} that
\begin{eqnarray*}&&R_{F_u}(2^{u+1}+2^u+2)-R_{E_u}(2^{u+1}+2^u+2)\\
&=&-\chi_{F_u}(2^{u+1}+2^u+1-M)+\chi_{F_u}(2^{u+1}+2^u+2-M)+\chi_A(2^{u+1}+2^u+2).\nonumber
\end{eqnarray*}
Noting that
$$R_{E_u}(2^{u+1}+2^u+2)=R_{V_u}(2^u)+R_{U_u,V_u}(2^u+1)$$
and
$$R_{F_u}(2^{u+1}+2^u+2)=R_{U_u}(2^u)+R_{U_u,V_u}(2^u+1)+2,$$
it follows from Lemma \ref{lem999999} that
$$R_{F_u}(2^{u+1}+2^u+2)-R_{E_u}(2^{u+1}+2^u+2)=2.$$

If $M=2^{u+1}+1+\sum\limits_{i=c}^{u-1}2^i$, then by
$$\chi_{F_u}(2^{u+1}+2^u+1-M)=\chi_{F_u}(2^c)=1$$
and
$$\chi_{F_u}(2^{u+1}+2^u+2-M)=\chi_{F_u}(2^c+1)=0,$$
we have $\chi_A(2^{u+1}+2^u+2)=3$, a contradiction.

If $M=2^{u+1}+1+\sum\limits_{i=0}^{t}2^i+\sum\limits_{i=t+2}^{u-1}2^i$, then by
$$\chi_{F_u}(2^{u+1}+2^u+1-M)=\chi_{F_u}(2^{t+1}+1)=0$$
and
$$\chi_{F_u}(2^{u+1}+2^u+2-M)=\chi_{F_u}(2^{t+1}+2)=0,$$
we have $\chi_A(2^{u+1}+2^u+2)=2$, a contradiction.

If $r=0$, then $(A-1)\cup (B-1)=\mathbb{N}\setminus \{m-1+mk:k\in \mathbb{N}\}$, $(A-1)\cap (B-1)=\emptyset$ and $R_{A-1}(n)=R_{B-1}(n)$ for every positive integer $n$. By the  proof of the case of $r\geq 1$, there exists a nonnegative integer $v$ such that $m=2^v+1$.

This completes the proof of Lemma \ref{55555}.
\end{proof}

\begin{lem}\label{lemy2} Let $l$ be a nonnegative integer and let

\begin{eqnarray*}
\varepsilon(l)=
\begin{cases}
1          &\mbox{ $l=0$},\\
2^{l-1}&\mbox{ $l\geq1$}.\\
\end{cases}
\end{eqnarray*}
We have

(i) $S_1(l)\cup T_1(l)=\mathbb{N}\setminus\{2^l+(2^l+1)k:k\in \mathbb{N}\},~S_1(l)\cap T_1(l)=\emptyset$
and $R_{S_1(l)}(n)=R_{T_1(l)}(n)$ for every positive integer $n$.

(ii) $(S_1(l)+1)\cup (T_1(l)+1)=\mathbb{N}\setminus\{(2^l+1)k:k\in \mathbb{N}\},~S_1(l)\cap T_1(l)=\emptyset$
and $R_{S_1(l)}(n)=R_{T_1(l)}(n)$ for every positive integer $n$.

(iii) $S_2(l)\cup T_2(l)=\mathbb{N}\setminus\{\varepsilon(l)+(2^{l}+1)k:k\in \mathbb{N}\},~S_2(l)\cap T_2(l)=\emptyset$
and $R_{S_2(l)}(n)=R_{T_2(l)}(n)$ for every positive integer $n$.
\end{lem}
\begin{proof}
Lemma \ref{lemy2} follows from Lemma \ref{lemy1} immediately.
\end{proof}

\section{Proof of Theorem \ref{thm2}}

In this section, we always assume that $1<m\le r-1$ and $A$ and $B$
are two sets such that $A\cup B=\mathbb{N}\setminus\{ r+mk :
k=0,1,\dots \}$, $A\cap B=\emptyset$, $0\in A$ and $R_{A}(n)=R_{B}(n)$ for every positive integer $n$.  We will derive a contradiction.

Firstly, we give the following lemmas.

\begin{lem}\label{lema1}  (i)  $r\notin U$ and $r+1\in U$;

(ii) If $2\nmid m$, then $m\notin U$ and $m+r\in U$.

(ii) If $2\mid m$ and $m<r-1$, then $m\notin U$.
\end{lem}

\begin{proof}  (i) By Lemma \ref{lem3}, we have $r\notin U$.
In view of Lemma \ref{lem3} with $n=r$, we have
$$\chi_{U}(1)=\chi_{U}(0)+\chi_A (r+1)-\chi_{U}(r+1).$$
Noting that $0\in U$ and $1\notin U$, we have
$$1=\chi_U (r+1)-\chi_{A}(r+1).$$
Hence, $r+1\notin A$ and $r+1\in U$.

(ii) If $2\nmid m$, then $\chi_{U}(m)+\chi_{U}(m-1)=1$. By
Lemma \ref{lem3} with $n=r+m-1$, we have
$$\chi_{U}(m)=\chi_{U}(m-1)+\chi_{A}(r+m)-\chi_{U}(r+m).$$
Noting that $r+m\notin A$, we have
\begin{eqnarray*}1&=&\chi_{U}(m)+\chi_{U}(m-1)\\
&=&2\chi_{U}(m-1)+\chi_{A}(r+m)-\chi_{U}(r+m)\\
&=& 2\chi_{U}(m-1)-\chi_{U}(r+m).\end{eqnarray*}
It follows that $\chi_{U}(m-1)=\chi_{U}(r+m)=1$. Hence $m\notin U, r+m\in U$.

(iii) If $2\mid m$, then $\chi_{U}(m)+\chi_{U}(m+1)=1$. By
$m<r-1$ and Lemma \ref{lem3} with $n=r+m$ we have
$$\chi_{U}(m+1)+\chi_{U}(1)=\chi_{U}(m)+\chi_{U}(0)+\chi_{A}(r+m+1)-\chi_{U}(r+m+1).$$
Noting that $0\in U$ and $1\notin U$, we have
\begin{eqnarray*}1&=&\chi_{U}(m)+\chi_{U}(m+1)+\chi_{U}(1)\\
&=&2\chi_{U}(m)+\chi_{U}(0)+\chi_{A}(r+m+1)-\chi_{U}(r+m+1)\\
&=&2\chi_{U}(m)+1+\chi_{A}(r+m+1)-\chi_{U}(r+m+1).\end{eqnarray*} That is,
$$-2\chi_{U}(m)=\chi_{A}(r+m+1)-\chi_{U}(r+m+1).$$
This implies that $\chi_{U}(m)=0$. Hence $m\notin U$.

This completes the proof of Lemma \ref{lema1}.
\end{proof}

\begin{lem}\label{lema2} If $k$ is a positive integer such that $r+2^k+1\in
U$, then  $r\ge 2^k-1$ and $m\le 2^k$. \end{lem}

\begin{proof} By Lemma \ref{lema1}, we have $r+1\in
U$. If $r< 2^k-1$, then $D_2(2^k+r+1)=1+D_2(r+1)$ is odd. So $r+2^k+1\notin U$, a contradiction. Hence, $r\ge 2^k-1$.

We prove $m\le 2^k$ by contradiction. Suppose that $m\ge 2^k+1$.
Then $r+2^k<r+m\le 2r-1$. By Lemma \ref{lem3} with $n=r+2^k$, we have
$$\chi_{U}(2^{k}+1)=\chi_{U}(2^{k})+\chi_{A}(r+2^k+1)-\chi_{U}(r+2^k+1).$$
Noting that $2^{k}+1\in U$, $r+2^k+1\in U$ and $2^{k}\notin U$, we have
$$\chi_{A}(r+2^k+1)=2,
$$
a contradiction. Hence, $m\le 2^k$.

This completes the proof of Lemma \ref{lema2}.
\end{proof}

\begin{lem}\label{2x} Let $u$ and $m$ be two positive integers with $m=2^u$. Then $2^u< r<2^{u+1}$.\end{lem}
\begin{proof}By $m\leq r-1$, we have
$2^u< r$. Let $k_1$ be the integer with $2^{k_1}\leq r<2^{k_1+1}$.
Then $k_1\ge u$. In view of $r\not\in
U$ and $D_2(2^{k_1}+r)=D_2(r)$, we have $2^{k_1}+r\not\in
U$. Suppose that $k_1>u$. By Lemma \ref{lem3} with
$n=r+2^{k_1}-1$,
\begin{eqnarray*} \sum_{r\le r+im\le n }  \chi_{U}
(n-r-im+1)
 =\sum_{r\le r+im\le n} \chi_{U} (n-r-im)+\chi_A
(n+1)-\chi_{U}(n+1)-\varepsilon,
\end{eqnarray*}
where $\varepsilon=1$ if $n=2r-1$, otherwise $\varepsilon=0$.
Noting that
\begin{eqnarray}\label{leqn5}
&&\sum_{0\le i\le 2^{k_1-u}-1} \chi_{U} (2^{k_1}-i 2^u)\nonumber\\
&=&\sum_{1\le i\le 2^{k_1-u}} \chi_{U} (2^{k_1}-i2^u)-1\nonumber\\
&=&\sum_{1\le i\le 2^{k_1-u}} \chi_{U} (2^{k_1-u}-i)-1\nonumber\\
&=&\sum_{0\le i\le 2^{k_1-u}-1} \chi_{U} (2^{k_1-u}-1-i)-1,
\end{eqnarray}
it follows from $2^{k_1}+r\not\in A$ and $2^{k_1}+r\not\in
U$ that
\begin{eqnarray}\label{eqn5}
&&\sum_{0\le i\le 2^{k_1-u}-1} \chi_{U} (2^{k_1-u}-1-i)-1\nonumber\\
&=&\sum_{0\le i\le 2^{k_1-u}-1 }  \chi_{U} (2^{k_1}-1-i
2^u)-\varepsilon.
\end{eqnarray}
Since
$$D_2(2^{k_1}-1-i 2^u)=k_1-D_2(i), \quad 0\le i\le 2^{k_1-u}-1,$$
$$D_2(2^{k_1-u}-1-i)=k_1-u-D_2(i), \quad
0\le i\le 2^{k_1-u}-1 $$ and
\begin{eqnarray*}&& |\{ i : 0\le i\le 2^{k_1-u}-1, 2\mid D_2(i)\} |\\
&=&\sum_{0\le 2i\le k_1-u} \binom{k_1-u}{2i}\\
&=&\sum_{0\le 2i+1\le k_1-u} \binom{k_1-u}{2i+1}\\
&=&|\{ i : 0\le i\le 2^{k_1-u}-1, 2\nmid D_2(i)\}
|,\end{eqnarray*} it follows that
\begin{eqnarray}\label{eqref1}&&\sum_{0\le i\le 2^{k_1-u}-1 }  \chi_{U}
(2^{k_1}-1-i 2^u)\nonumber\\
&=& |\{ i : 0\le i\le 2^{k_1-u}, D_2(i)\equiv k_1\pmod 2 \} |\nonumber\\
&=& |\{ i : 0\le i\le 2^{k_1-u}, D_2(i)\equiv k_1-u\pmod 2 \} |\nonumber\\
&=& \sum_{0\le i\le 2^{k_1-u}-1} \chi_{U} (2^{k_1-u}-1-i).
\end{eqnarray}
Thus, by \eqref{eqn5}, we have $\varepsilon=1$.  So $r=2^{k_1}$. By Lemma \ref{lem3} with $n=r+m-1=2^{k_1}+2^u-1$, we have
$$\chi_{U}(2^u)=\chi_{U}(2^u-1)+\chi_{A}(2^{k_1}+2^u)-\chi_{U}(2^{k_1}+2^u).$$
Noting that $2^u\not\in U$, $2^{k_1}+2^u\in U$ and $2^{k_1}+2^u\notin A$, we have
$\chi_{U}(2^u-1)=1$. Hence, $2\mid u$. So $u\geq2$.
For any integer $s$ with $0\leq s\leq m-2$, by Lemma \ref{lem3} with $n=r+s$, we have
$$\chi_{U}(s+1)=\chi_{U}(s)+\chi_{A}(r+s+1)-\chi_{U}(r+s+1).$$
Noting that $$\chi_{U}(r+s+1)=\chi_{V}(s+1)=1-\chi_{U}(s+1),$$
we have
\begin{equation}\label{eqabc0}\chi_{U}(s)+\chi_{A}(r+s+1)=1.
\end{equation}
For any integer $s$ with $0\leq s\leq m-2$, by Lemma \ref{lem3} with $n=r+m+s$, we have
$$\chi_{U}(m+s+1)+\chi_{U}(s+1)=\chi_{U}(m+s)+\chi_{U}(s)+\chi_{A}(r+m+s+1)-\chi_{U}(r+m+s+1).$$
Noting that $$\chi_{U}(m+s+1)+\chi_{U}(s+1)=1$$
and $$\chi_{U}(m+s)+\chi_{U}(s)=1,$$
we have
\begin{equation}\label{eqabc2}
\chi_U(s+1)=\chi_A(r+m+s+1).
\end{equation}
By \eqref{eqabc0} and \eqref{eqabc2} and $2\mid u$, we have
$$A(r+2m)=(U_{k_1}\cup(2^{k_1}+1+V_{u})\cup(2^{k_1}+2^u+U_{u}))\setminus\{2^{k_1}+2^u\},$$
$$B(r+2m)=(V_{k_1}\cup(2^{k_1}+1+U_{u})\cup(2^{k_1}+2^u+V_{u}))\setminus\{2^{k_1}+2^u\}.$$
Now we employ Lemma
\ref{lem2} with $C=E_{k_1}$, $D=F_{k_1}$ and $L=r+m$, where $E_{k_1}$ and $F_{k_1}$ are as
in Lemma \ref{lem1w}. Noting that $2^{k_1}+2^u+1\in B$, $2^{k_1}+2^u+1\in E_{k_1}$, $2^{k_1}+2^u+2\in B$ and $2^{k_1}+2^u+2\in F_{k_1}$. Hence
$$R_{T,C\setminus (A\cup T(L,2r+2^u+1)}(2r+2^u+1)-R_{T,D\setminus (B\cup T(L,2r+2^u+1))}(2r+2^u+1)=1$$
and
$$R_{T,C\setminus (A\cup T(L,2r+2^u+1)}(2r+2^u+2)-R_{T,D\setminus (B\cup T(L,2r+2^u+1))}(2r+2^u+2)=0.$$

In view of Lemma \ref{lem2} with $n=2r+2^u+1$, we have
\begin{eqnarray}\label{eq16a}&&1+R_{F_{k_1}, T(r+m, 2r+2^u+1)\cap F_{k_1} }(2r+2^u+1)-R_{E_{k_1}, T(r+m, 2r+2^u+1)\cap
E_{k_1} }(2r+2^u+1)\nonumber\\
&=&R_{F_{k_1}, T(r+m, 2r+2^u+1)\cap F_{k_1}}(2r+2^u+1)-R_{E_{k_1}, T(r+m, 2r+2^u+1)\cap E_{k_1} }(2r+2^u+1)\nonumber\\
&&+\chi_A (2r+2^u+2)-\chi_{E_{k_1}}(2r+2^u+2).
\end{eqnarray}
Noting that $u\geq2$, we have
\begin{eqnarray*}&& R_{F_{k_1}, T(r+m,2r+2^u+1)\cap
F_{k_1} }(2r+2^u+1)-R_{E_{k_1}, T(r+m,2r+2^u+1)\cap E_{k_1} }(2r+2^u+1)\\
 &=& \sum_{t\in T(r+m,2r+2^u+1)} \chi_{F_{k_1}} (2r+2^u+1-t)\chi_{F_{k_1}}
 (t)\\
 &&-\sum_{t\in T(r+m,2r+2^u+1)} \chi_{E_{k_1}} (2r+2^u+1-t)\chi_{E_{k_1}}
 (t)\\
&=& \sum_{t\in T(r+m,2r+2^u+1)}(1- \chi_{E_{k_1}} (2r+2^u+1-t))(1-\chi_{E_{k_1}}(t))\\
&&-\sum_{t\in T(r+m,2r+2^u+1)} \chi_{E_{k_1}} (2r+2^u+1-t)\chi_{E_{k_1}}
 (t)\\
&=& \sum_{t\in T(r+m,2r+2^u+1)} 1-(\chi_{E_{k_1}} (2r+2^u+1-t)+\chi_{E_{k_1}}
 (t)).
\end{eqnarray*}
Similarly,
\begin{eqnarray*}&&
R_{F_{k_1}, T(r+m,2r+2^u+1)\cap
F_{k_1} }(2r+2^u+2)-R_{E_{k_1}, T(r+m,2r+2^u+1)\cap E_{k_1} }(2r+2^u+2)\\
&=& \sum_{t\in T(r+m,2r+2^u+1)} 1-(\chi_{E_{k_1}} (2r+2^u+2-t)+\chi_{E_{k_1}}
 (t)).
\end{eqnarray*}
It follows from \eqref{eq16a} that
\begin{eqnarray*}&&1+\sum_{t\in T(r+m,2r+2^u+1)} \chi_{E_{k_1}} (2r+2^u+2-t)\nonumber\\
&=&\sum_{t\in T(r+m,2r+2^u+1)}\chi_{E_{k_1}} (2r+2^u+1-t)+\chi_A (2r+2^u+2)-\chi_{E_{k_1}}(2r+2^u+2),
\end{eqnarray*}
that is
\begin{eqnarray*}&&1+\sum\limits_{i=0}^{2^{k_1-u}} \chi_{E_{k_1}} (2^{k_1}-i2^u+2)\nonumber\\
&=&\sum\limits_{i=0}^{2^{k_1-u}} \chi_{E_{k_1}} (2^{k_1}-i2^u+1)+\chi_A (2r+2^u+2)-\chi_{E_{k_1}}(2r+2^u+2),
\end{eqnarray*}
which implies that
\begin{eqnarray*}&&1+\chi_{E_{k_1}} (2^{k_1}+2)\nonumber\\
&=&\chi_{E_{k_1}} (2^{k_1}+1)+\chi_A (2r+2^u+2)-\chi_{E_{k_1}}(2r+2^u+2).
\end{eqnarray*}
Noting that $\chi_{E_{k_1}} (2^{k_1}+2)=1$ and $\chi_{E_{k_1}} (2^{k_1}+1)=\chi_{E_{k_1}}(2r+2^u+2)=0$, we have $\chi_A (2r+2^u+2)=2$, a contradiction.
Hence, $u=k_1$. Noting that $m\leq r-1$, we have $2^u< r<2^{u+1}$.

This completes the proof of Lemma \ref{2x}.
\end{proof}

\begin{lem}\label{lema4} If $d$ is a positive integer with $m+2^d< r-1$, $m+2^d\notin U$, and $m\not=2^d$,
then $m$ is odd, $m+2^d+1\notin U$ and $r+m+2^d+1\notin U$.\end{lem}

\begin{proof} By Lemma \ref{lema1}, $m\notin U$. If
$m<2^d$, then $D_2(m+2^d)=D_2(m)+1$ is even, a contradiction with
$m+2^d\notin U$. So $m\ge 2^d$. Noting that $m\not= 2^d$, we have
$m>2^d$, and so $m+2^d<2m$. In view of Lemma \ref{lem3} with $n=r+m+2^d< 2r-1$, we have
$$\chi_{U}(m+2^d+1)+\chi_{U}(2^d+1)
=\chi_{U}(m+2^d)+\chi_{U}(2^d) +\chi_A
(r+m+2^d+1)-\chi_{U}(r+m+2^d+1).$$
 Noting that $m+2^d,2^d\notin U$ and $2^d+1\in U$, we have
$$\chi_{U}(m+2^d+1)=0,~~\chi_{U}(r+m+2^d+1)=0.$$
If $m$ is even, then by $m+2^d\notin
U$, we have $m+2^d+1\in U$, a contradiction.
Hence $m$ is odd.

This completes the proof of Lemma \ref{lema4}.
\end{proof}

\begin{proof}[Proof of Theorem \ref{thm2}] Assume that $2\leq m\leq r-1$. We will consider the following two cases.

{\bf Case 1:} $2\nmid r$. Let
$$r=\sum\limits_{i=0}^{j}2^i+2^{j+2}R$$
with $j\ge 0$ and $j,R\in \mathbb{N}$. By $r\notin U$ and $r+1\in
U$, we have $2\nmid
 j$ and $2\nmid D_2(R)$. Hence $D_2(R)\geq1$. Since $$r+1+2^{j+1}=2^{j+2}+2^{j+2}R,\quad
 r+1+2^{j+2}=2^{j+1}+2^{j+2}+2^{j+2}R,$$
there exists integer $k\in \{ j+1, j+2\} $ such that $r+1+2^k\in U$. By
Lemma \ref{lema2}, $r\ge 2^k-1$ and $m\le 2^k$. Let $u$ be a positive integer with $2^u\leq m< 2^{u+1}$. Then $u\leq k$.  If $m=2^u$, then by Lemma \ref{2x}, we have $2^{u}< r<2^{u+1}$. Noting that $R\geq1$, we have $r\geq2^k$, and so $k< u+1$. So $k=u$. It follows from $R\geq1$ that $k=j+2$, which implies that $m=2^{j+2}$ and $r=\sum\limits_{i=0}^{j}2^i+2^{j+2}$.

For any integer $0\leq t\leq m-1$, by Lemma \ref{lem3} with $n=r+t$, we have
\begin{equation}\label{a1}\chi_{U}(t+1)=\chi_{U}(t)+\chi_{A}(r+t+1)-\chi_{U}(r+t+1).
\end{equation}
If $0\leq t\leq 2^{j+1}-1$, then $D_2(r+t+1)=D_2(t)$. If $2^{j+1}\leq t\leq m-1$, then $D_2(r+t+1)=D_2(t)$.
Hence, for any integer $0\leq t\leq m-1$, we have $\chi_{U}(t)-\chi_{U}(r+t+1)=0$, it follows from \eqref{a1} that
\begin{equation}\label{a2}\chi_{U}(t+1)=\chi_{A}(r+t+1).
\end{equation}
For any integer $0\leq s\leq 2^{j+1}-2$, by Lemma \ref{lem3} with $n=r+m+s$, we have
\begin{equation}\label{a3}\chi_{U}(m+s+1)+\chi_{U}(s+1)=\chi_{U}(m+s)+\chi_{U}(s)+\chi_{A}(r+m+s+1)-\chi_{U}(r+m+s+1)-\varepsilon,
\end{equation}
where $\varepsilon=1$ if $n=r+m+2^{j+1}-2=2r-1$, otherwise $\varepsilon=0$.
Since $\chi_{U}(m+s+1)+\chi_{U}(s+1)=1$ and $\chi_{U}(m+s)+\chi_{U}(s)=1$, it follows from \eqref{a3} that $\chi_{A}(2r)=1$ and
\begin{equation}\label{a4}\chi_{A}(r+m+s+1)=\chi_{U}(r+m+s+1)=\chi_{U}(s),~~~\text{if} ~~~0\leq s\leq 2^{j+1}-3.
\end{equation}
It follows from \eqref{a2}, \eqref{a4} and $2\nmid j$ that
$$A(2r)=(U_{j+2}\cup(2^{j+2}+V_{j+1})\cup(r+U_{j+2})\cup(r+m+1+U_{j+1})\cup\{2r\})\setminus\{r,2r+1\},$$
$$B(2r)=(V_{j+2}\cup(2^{j+2}+U_{j+1})\cup(r+V_{j+2})\cup(r+m+1+V_{j+1}))\setminus\{r,2r\}.$$
Since
\begin{eqnarray*}
&&R_{A(2r)}(2r+1)\\
&=&R_{U_{j+2},U_{j+2}}(2^{j+2}+2^{j+1})+R_{U_{j+2},V_{j+1}}(2^{j+1})+R_{U_{j+2}}(1)\\
&&-\chi_{U_{j+2}}(1)+R_{U_{j+2},U_{j+1}}(2^{j+1}-1)-\chi_{U_{j+2}}(0)\\
&=&2R_{U_{j+2}}(2^{j+2}+2^{j+1})+\chi_{U_{j+2}}(2^{j+1}+2^{j})+R_{U_{j+1},V_{j+1}}(2^{j+1})+2R_{V_{j+1}}(0)\\
&&+\chi_{V_{j+1}}(0)+R_{U_{j+2}}(1)-\chi_{U_{j+2}}(1)+2R_{U_{j+1}}(2^{j+1}-1)-\chi_{U_{j+2}}(0)\\
&=&2R_{U_{j+2}}(2^{j+2}+2^{j+1})+1+R_{U_{j+1},V_{j+1}}(2^{j+1})+2R_{U_{j+1}}(2^{j+1}-1)-1\\
&=&2R_{U_{j+2}}(2^{j+2}+2^{j+1})+R_{U_{j+1},V_{j+1}}(2^{j+1})+2R_{U_{j+1}}(2^{j+1}-1),
\end{eqnarray*}
\begin{eqnarray*}
&&R_{B(2r)}(2r+1)\\
&=&R_{V_{j+2},V_{j+2}}(2^{j+2}+2^{j+1})+R_{V_{j+2},U_{j+1}}(2^{j+1})+R_{V_{j+2}}(1)\\
&&-\chi_{V_{j+2}}(1)+R_{V_{j+2},V_{j+1}}(2^{j+1}-1)-1\\
&=&2R_{V_{j+2}}(2^{j+2}+2^{j+1})+R_{V_{j+1},U_{j+1}}(2^{j+1})+2R_{V_{j+1}}(2^{j+1}-1)-1,
\end{eqnarray*}
Hence, by Lemma \ref{lem999999}, we have
$$R_{A(2r)}(2r+1)-R_{B(2r)}(2r+1)=1,$$
it follows that
$$0=R_{A}(2r+1)-R_{B}(2r+1)=R_{A(2r)}(2r+1)-R_{B(2r)}(2r+1)+\chi_A(2r+1)=1+\chi_A(2r+1),$$
a contradiction.

Hence, $m$ is  not a power of 2, we have $m<2^k$. Then $u\leq k-1$. If $m$ is odd, then let
$$m=\sum_{i=0}^d 2^i +2^{d+2}M$$ with $d\ge 0$ and $d,M\in
\mathbb{N}$. By Lemma \ref{lema1}, $m\notin U$, $r+m\in U$. Noting that $m\leq r-1$, $m$ and $r$ are both odds, we have $m\leq r-2$.
Noting that $2\mid r+m$, we have $r+m+1 \notin U$. By Lemma
\ref{lem3} with $n=r+m<2r-1$, we have
$$\chi_{U}(m+1)+\chi_{U}(1)
=\chi_{U}(m)+\chi_{U}(0)+\chi_{A}(r+m+1)-\chi_{U}(r+m+1).$$
Since $$\chi_{U}(m)=0,\quad \chi_{U}(0)=1,\quad
\chi_{U}(1)=0,\quad \chi_{U}(r+m+1)=0,$$ it follows that
$\chi_{U}(m+1)-\chi_{A}(r+m+1)=1$. So $m+1\in U$. By $m\notin U$,
$m+1\in U$ and
$$D_2(m+1)=D_2(2^{d+1}+2^{d+2}M)=1+D_2(M)=D_2(m)-d,$$  we have $2\nmid d$ and $2\nmid D_2(M)$.
Hence $d\ge 1$ and $M\ge 1$. Noting that $u\leq k-1$, we have $u\leq j+1$. Since $M\geq1$ and $d\geq1$, it follows that $m<2^{j+2}-1-2$, and so $m+2^u<2^{j+2}-1-2+2^{j+1}\leq r-2$. Noting that $m+2^u\notin U$ from $D_2(m+2^u)=D_2(m)$ and $m\not\in U$. By Lemma \ref{lema4}, we have $m+2^u+1\notin U$. Since $M\geq1$,  it follows that $2^u<m+1<2^{u+1}$. By $m+1\in U$, we have $m+2^u+1\in U$, a contradiction.
Hence $2\mid m$. If $m=r-1$, then by Lemma
\ref{lem3} with $n=r+m=2r-1$, we have
$$\chi_{U}(r)+\chi_{U}(1)
=\chi_{U}(r-1)+\chi_{U}(0)+\chi_{A}(2r)-\chi_{U}(2r)-1.$$
Since $$\chi_{U}(r)=0,\quad \chi_{U}(1)=0,\quad \chi_{U}(r-1)=1,\quad
\chi_{U}(0)=1,\quad \chi_{U}(2r)=0,$$ it follows that
$\chi_{A}(2r)=-1$, a contradiction. So
$m<r-1$. Let
$$m=\sum_{i=c}^d 2^i +2^{d+2}M$$ with $d\ge c\ge 1$ and $c,d,M\in
\mathbb{N}$. By Lemma \ref{lema1}, $m\notin U$. It follows from
$$D_2(m+2^d)=D_2(m)$$ that $m+2^d\notin U$. By $2\mid m$, $m\neq 2^d$
and Lemma \ref{lema4}, we have $m+2^d\ge r-1$. If $M\ge 1$, then by
$m<2^k$ and $r\geq 2^k-1$, we have
$$m+2^d=\sum_{i=c}^{d-1} 2^i+2^{d+1}+2^{d+2}M\le 2^k-2^{c+1}\le 2^k-2^2<r-1,$$ a contradiction. So
$M=0$. By $m<2^k$ we have $d\le k-1\le j+1$. Noting that $D_2(R)\geq1$, we have
$$m+2^d=\sum\limits_{i=c}^{d-1} 2^i+2^{d+1}\le \sum_{i=c}^{j} 2^i+2^{j+2}.$$
If $m+2^d=\sum\limits_{i=c}^{j} 2^i+2^{j+2}$, then by $m+2^d\not\in U$ and $2\nmid j$, we have $2\mid c$, and so $c\geq 2$. Thus,
$$m+2^d\leq\sum\limits_{i=2}^{d-1} 2^i+2^{d+1}<r-1,$$
a contradiction.
If $m+2^d<\sum\limits_{i=c}^{j} 2^i+2^{j+2}$, then $m+2^d<\sum\limits_{i=c}^{j} 2^i+2^{j+2}\leq r-1,$
a contradiction.

 {\bf Case 2: } $2\mid r$. Let
 $$r=\sum_{i=s}^t 2^i +2^{t+2}R,$$
 where $s,t,R$ are integers with $t\ge s\ge 1$ and $R\in \mathbb{N}$.
Since
$$r+2^s+1=1+2^{t+1}+2^{t+2}R,\quad
r+2^{s+1}+1=1+2^s+2^{t+1}+2^{t+2}R,$$ there exists $k\in \{ s,
s+1\} $ such that $r+2^k+1\in U$. By $2\mid r$ and Lemma \ref{lema2}, we have $r\ge
2^k$ and $m\le 2^k$.  Let $u$ be the integer with $2^u\le m<2^{u+1}$. Then $u\leq k\leq s+1$. If $m=2^u$, then by Lemma \ref{2x}, we have $2^u<r<2^{u+1}$. Noting that $r\geq 2^k$, we have $k<u+1$, and so $u=k$. If $k=s+1$, then $r=2^s+2^{s+1}\in U$, a contradiction. So $u=k=s$, $2^s<r<2^{s+1}$, a contradiction. Hence, $m$ is a not power of $2$, then $u\leq k-1\leq s$.  It is
clear that $D_2(m+2^u)=D_2(m)$. If $r=2^s$, then $r+2^s+1\in U$, and so $k=s$ and $u\leq s-1$. If $u=s-1$ and $m$ is even, then $s\geq2$, by Lemma \ref{lema4}, we have $$m=\sum_{i=c}^{s-1}2^{i},$$
where $c\geq1$. By Lemma \ref{lem3} with $n=r+m-1$, we have
\begin{eqnarray}\label{1eqc1j}&&\chi_{U}(m)=\chi_{U}(m-1)+\chi_{A}(r+m)-\chi_{U}(r+m).\end{eqnarray}
Since $$\chi_{U}(m)=0,\chi_{A}(r+m)=0,\chi_{U}(r+m)=1,$$ it follows from \eqref{1eqc1j} that
$$\chi_{U}(m-1)=1.$$
Since $m-1\in U$, then $2\mid c$. Since $m\notin U$, then $2\nmid s$, and so $s\geq3$. For $n=r+2m-1<2(r+m)-1$, similar to the proof of Lemma \ref{lemh}, we have
\begin{equation*}\label{x1}\chi_{F_{s}}(m-1)=\chi_{F_{s}}(m)+\chi_A(r+2m)-\chi_{E_{s}}(r+2m).\end{equation*}
Since $$\chi_{E_{s}}(r+2m)=\chi_{E_{s}}(2^{s+1}+\sum_{i=c+1}^{s-1}2^{i})=\chi_V(\sum_{i=c+1}^{s-1}2^{i}-1)=0,$$
it follows from $\chi_{F_{s}}(m-1)=0$ and $\chi_{F_{s}}(m)=1$ that
$\chi_A(r+2m)=-1$, a contradiction. So $m$ is odd. Similarly, for $n=r+2m-1$, we have
\begin{equation*}\label{x1}\chi_{F_{s}}(m-1)=\chi_{F_{s}}(m)+\chi_A(r+2m)-\chi_{E_{s}}(r+2m).\end{equation*}
Since $$\chi_{E_{s}}(r+2m)=\chi_{E_{s}}(2^{s+1}+2m-2^s)=\chi_V(2m-2^s-1)=0,$$
it follows from $\chi_{F_{s}}(m-1)=0$ and $\chi_{F_{s}}(m)=1$ that
$\chi_A(r+2m)=-1$, a contradiction. So $u\leq s-2$. Since $m\geq2$, then $u\geq1$ and $s\geq3$. Then $m+2^u<2^{u+1}+2^u\le 2^{s-1}+2^{s-2}<r-1$.
If $r\geq2^s+2^{s+1}$, then by $r\not\in U$, we have $r\geq2^s+2^{s+1}+2^{s+2}$. So $m+2^u<2^{u+1}+2^u\le 2^{s+1}+2^s<r-1$. Noting that $m\geq2$, we have $u\geq1$. By Lemma \ref{lema1}, we have $m\notin U$, and so $m+2^u\notin U$ since $D_2(m+2^u)=D_2(m)$. It follows from Lemma \ref{lema4} that $m$ is odd, $m+2^u+1\notin U$ and $r+m+2^u+1\notin U$. Noting that $m$ is odd, we may write
$$m=\sum_{i=0}^d 2^i+2^{d+2}M.$$
If $M=0$, then $u=d$ and $m+2^u+1=2^d+2^{d+1}\in U$, a
contradiction. So $M\ge 1$.
If $u\leq s-2$, then by $M\geq1$, we have $2^u<m+1<2^{u+1}$, and so $m+2^u+1<2^{u+1}+2^{u}\leq 2^{s-1}+2^{s-2}<r$. Hence,
$$D_2(r+m+2^u+1)=D_2(r)+D_2(m+2^u+1),$$
it follows from $r\notin U$ and $m+2^u+1\notin U$ that $r+m+2^u+1\in U$, a contradiction with $r+m+2^u+1\notin U$. So $u\in \{s-1,s\}$ and $r\geq 2^s+2^{s+1}$. If $u=s-1$, then
$$D_2(r+m+2^u+1)=D_2(1+m-2^{s-1}+2^{t+1}+2^{t+2}R)=D_2(R)+D_2(m+2^u+1),$$
it follows from $r+m+2^u+1\notin U$ and $m+2^u+1\notin U$ that $2\mid D_2(R)$. By Lemma \ref{lem3} with $n=r+2m-1$, we have
$$\chi_{U}(2m)+\chi_{U}(m)=\chi_{U}(2m-1)+\chi_{U}(m-1)+\chi_{A}(r+2m)-\chi_{U}(r+2m).$$
Since $\chi_{U}(2m)=\chi_{U}(m)=0,\chi_{U}(2m-1)+\chi_{U}(m-1)=1$,
it follows that $\chi_{U}(r+2m)=1$. Noting that $u=s-1$, we have
$$D_2(r+2m)=D_2(m)+D_2(R),$$
it follows from $m\notin U$ and $2\mid D_2(R)$ that $r+2m\notin U$, a contradiction. So $u=s$. It follows from $u\leq k-1\leq s$ that $k=s+1$. By definition of $k$, we have
$r+1+2^{s+1}\in U$, and so $2\nmid D_2(R)$. Noting that $m+2^d<r-1$ and $m+2^d\notin U$, by Lemma \ref{lema4}, we have $m+2^d+1\notin U$.
Let $m+2^d+1=2^s+M_1$. By $u=s$, we have
$$m+2^d+1=2^d+2^{d+1}+2^{d+2}M\le 2^d+2^{d+1}+\cdots
+2^s=2^{s+1}-2^d.$$ It follows that
$$D_2(m+2^d+1)=D_2(2^s+M_1)=1+D_2(M_1).$$ In view of $m+2^d+1\notin
U$, we have $2\mid D_2(M_1)$. Since $M_1<2^s<2^{t+1}$ and
$$r+m+2^d+1=r+2^s+M_1=2^{t+1}+2^{t+2}R+M_1,$$
we have $D_2(r+m+2^d+1)=1+D_2(R)+D_2(M_1)$. It follows from $2\nmid
D_2(R)$ and $2\mid D_2(M_1)$ that $2\mid D_2(r+m+2^d+1)$, a
contradiction with $r+m+2^d+1\notin U$.

Up to now, we have proved $m\geq r$. Theorem \ref{thm2} follows from Lemma \ref{999999}, Lemma \ref{55555} and Lemma \ref{lemy2} immediately.
\end{proof}

\end{document}